\documentclass[a4paper,12pt]{amsart}

\makeindex

\usepackage{amsfonts,graphics,amsmath,amsthm,amsfonts,amscd,amssymb,amsmath,latexsym,euscript, enumerate}
\usepackage{epsfig}
\usepackage{flafter}
\usepackage[all,cmtip,line]{xy}
\usepackage{array}
\usepackage[english]{babel}
\usepackage{overpic}
\usepackage{subfig}
\usepackage{multirow}
\usepackage{microtype}
\usepackage{wrapfig}
\usepackage{ulem}
\allowdisplaybreaks

\usepackage[dvipsnames,svgnames,table]{xcolor}
\usepackage{graphicx}
\usepackage[hyphens]{url}
\usepackage{hyperref}

\newtheorem{theorem}{Theorem}[section]
\newtheorem{lemma}[theorem]{Lemma}
\newtheorem{proposition}[theorem]{Proposition}

\newtheorem*{theorem*}{Theorem}

\theoremstyle{plain}
\newtheorem{corollary}[theorem]{Corollary}

\theoremstyle{definition} 
\newtheorem{definition}[theorem]{Definition}
\newtheorem{definition-lemma}[theorem]{Definition-Lemma}

\newtheorem{example}[theorem]{Example}

\numberwithin{equation}{section}

\newcommand{\C}{\mathbb{C}}
\newcommand{\R}{\mathbb{R}}

\newcommand{\Q}{\mathbb{Q}}

\DeclareMathOperator{\codim}{codim}

\newcommand{\bm}{\mathbf B_-}  
\newcommand{\B}{\mathbf B}
\newcommand{\bp}{\mathbf B_+}  
\def\pnklt{\operatorname{pNklt}}
\def\gnklt{\operatorname{gNklt}}

\def\mult{\operatorname{mult}}

\def\Supp{\operatorname{Supp}}
\def\Exc{\operatorname{Exc}}
\def\NE{\operatorname{NE}}

\def\codim{\operatorname{codim}}

\def\pnklt{\operatorname{pNklt}}
\def\nklt{\operatorname{Nklt}}
\def\nnef{\operatorname{NNef}}

\newcommand{\pp}{P_{\sigma}}
\newcommand{\np}{N_{\sigma}}

\usepackage{mathtools}

\DeclarePairedDelimiterX{\norm}[1]{\lVert}{\rVert}{#1}

\newcommand{\coleq}{\vcentcolon=}

\title[On minimal model program and Zariski decomposition of potential triples]
{On minimal model program and Zariski decomposition of potential triples}

\begin{document}

\author[S.~Choi]{Sung Rak Choi}
\author[S.~Jang]{Sungwook Jang}
\author[D.-W.~Lee]{Dae-Won Lee}
\address{Department of Mathematics, Yonsei University, 50 Yonsei-ro, Seodaemun-gu, Seoul 03722, Republic of Korea}
\email{sungrakc@yonsei.ac.kr}
\address{Center for Complex Geometry, Institute for Basic Science, 34126 Daejeon, Republic of Korea}
\email{swjang@ibs.re.kr}
\address{Department of Mathematics, Ewha Womans University, 52 Ewhayeodae-gil, Seodaemun-gu, Seoul 03760, Republic of Korea}
\email{daewonlee@ewha.ac.kr }

\thanks{The authors are partially supported by Samsung Science and Technology Foundation under Project Number SSTF-BA2302-03. D. Lee is partially supported by Basic Science Research Program through the National Research Foundation of Korea (NRF) funded by the Ministry of Education (No. RS-2023-00237440).}


\begin{abstract}
  In this paper, we investigate properties of potential triples $(X,\Delta,D)$ which consists of a pair $(X,\Delta)$ and a pseudoeffective $\R$-Cartier divisor $D$. In particular, we  show that if $D$ admits a birational Zariski decomposition, then one can associate a generalized pair structure to the potential triple $(X,\Delta,D)$. Moreover, we can run the generalized MMP on $(K_X+\Delta+D)$ as special cases. As an application, we also show that for a pklt pair $(X,\Delta)$, if $-(K_X+\Delta)$ admits a birational Zariski decomposition with $\mathrm{NQC}$ positive part, then there exists a $-(K_X+\Delta)$-minimal model.
\end{abstract}

\maketitle



\section{Introduction}\label{sect:intro}

It is well known that understanding the geometry of the generalized pairs has been a driving force recently, especially in the field of birational geometry. In this paper, we explore the geometry of the generalized pairs in even more general settings. A potential triple $(X,\Delta,D)$ consists of a pair $(X,\Delta)$ equipped with another divisor $D$ which we assume to be $\R$-Cartier and only pseudoeffective. Recall that if $D$ is the trace of a b-nef divisor, then the triple coincides with the generalized pair and if $D=0$, then the triple is nothing but the usual pair. By defining the potential log discrepancy as the usual log discrepancy of $(X,\Delta)$ subtracted by the asymptotic divisorial valuation of $D$, we can extend the notions of singularities of pairs to define, for example, potentially klt, potentially lc and so on. The generalizations of log discrepancies as well as other notions associated to pairs also allow us to consider a variation of the minimal model program.  See Section \ref{sect:prelim}. 

The minimal model program (MMP for short) on the generalized pairs has been studied extensively in a series of papers \cite{CHLX},\cite{CT},\cite{HL},\cite{TX}, etc. Recently, by analyzing the algebraically integrable foliation structures on the generalized pairs, the MMP for the glc pairs has been established by \cite{CHLX}. In this paper, we aim to obtain some necessary conditions for the MMP on the potential triples. By definition, we require such MMP to be a birational contraction $\varphi\colon X\dashrightarrow Y$ which is $(K_X+\Delta+D)$-negative. The resulting model is also a potential triple $(Y,\Delta_Y,D_Y)$ which is not necessarily a generalized pair. Of course, it is expected that the map $\varphi$ is a composition of divisorial contractions and flips. If this is the case and once the proper transform of $D$ in an intermediate step becomes nef, then from then on, the rest of the steps in the MMP coincide with the MMP on the generalized pairs. We will elaborate on this circumstance in Section \ref{sect:prelim} and Theorem \ref{thm:run mmp}. The following is the main result of this paper.

\begin{theorem}\label{thm:main1}
  Let $(X,\Delta,D)$ be a potential triple such that $D$ admits a birational Zariski decomposition $f^{\ast}D=P+N$ for some birational morphism $f\colon Y\to X$. Then the followings hold:
  \begin{enumerate}[(1)]
    \item If $(X,\Delta,D)$ is a plc triple, then $\pnklt(X,\Delta,D)$ is Zariski closed.
    \item If $(X,\Delta,D)$ is a $\Q$-factorial plc triple, then one can run the $(K_X+\Delta+D)$-MMP.
  \end{enumerate}
\end{theorem}

In the proof, we will use the fact that we can associate a generalized pair $(X,(\Delta+f_{\ast}N)+f_{\ast}P)$ if $D$ admits a birational Zariski decomposition as in Theorem \ref{thm:main1}. For the definitions of potentially non-klt locus $\pnklt(X,\Delta,D)$ and  $(K_X+\Delta+D)$-MMP, see Section \ref{sect:prelim}.
If $D=-(K_X+\Delta)$ in a potential triple $(X,\Delta,D)$, then we just call $(X,\Delta)$ a potential pair.

As an application of Theorem \ref{thm:main1}, we obtain the following result.

\begin{corollary}[cf. {\cite[Theorem 1.3]{Jan}}]\label{cor:main1}
  Let $(X,\Delta)$ be a $\Q$-factorial pklt pair such that $-(K_X+\Delta)$ admits a birational Zariski decomposition with $\mathrm{NQC}$ positive part. Then there exists a $-(K_X+\Delta)$-minimal model.
\end{corollary}

If the augmented base locus $\bp(D)$ does not contain any plc centers of $(X,\Delta,D)$, then we can run the $(K_X+\Delta+D)$-MMP.

\begin{theorem}\label{thm:main2}
Let $(X,\Delta,D)$ be a $\Q$-factorial plc triple with a big $\Q$-Cartier divisor $D$. If no plc centers of $(X,\Delta,D)$ are contained in $\bp(D)$, then there exists an effective divisor $D'\sim_{\Q} D$ such that $(X,\Delta+D')$ is an lc pair and $\pnklt(X,\Delta,D)=\nklt(X,\Delta+D')$. Moreover, we can run the $(K_X+\Delta+D)$-MMP.
\end{theorem}

We recall some related results. If $(X,\Delta)$ is an lc pair and $D=K_X+\Delta$ is a big divisor such that $\bp(D)$ does not contain any lc centers of $(X,\Delta)$, then the main result of \cite{BH} tells us that there exists a good minimal model of $(X,\Delta)$. If $D=-(K_X+\Delta)$ is big and satisfies the conditions in Theorem \ref{thm:main2}, then by \cite{CJL} there exists a good $-(K_X+\Delta)$-minimal model. The condition in Theorem \ref{thm:main2} implies that for any sufficiently small $\varepsilon>0$, $(X,\Delta,(1+\varepsilon)D)$ is also plc. Furthermore, by \cite[Lemma 4.6]{CJK} we have $\pnklt(X,\Delta,D)=\pnklt(X,\Delta,(1+\varepsilon)D)$. Thus philosophically, the $-(K_X+\Delta)$-MMP to the good $-(K_X+\Delta)$-minimal model from $X$ is supposedly the birational map given by the $(K_X+\Delta+(1+\varepsilon)D)$-MMP in Theorem \ref{thm:main2}.  Note that such MMP coincides with $-(K_X+\Delta)$-MMP. At this moment, however, the existence of $(K_X+\Delta+D)$-minimal model for general $D$ is unknown and yet to be studied further.

The rest of this paper is organized as follows. In Section \ref{sect:prelim}, we recall the definitions and preliminary results regarding generalized pairs, minimal model program, Zariski decomposition, potential triples and asymptotic base loci. In Section \ref{sect:main}, we prove several properties on potential triples $(X,\Delta,D)$ with $D$ admitting a birational Zariski decomposition. In particular, we compare potential triples with generalized pairs. Also, we prove the main theorems, Theorem \ref{thm:main1} and Theorem \ref{thm:main2}.

\section*{Acknowledgement}
We would like to thank Donghyeon Kim for pointing out a mistake in the earlier version of this paper.

\section{Preliminaries}\label{sect:prelim}

Throughout the paper, we work over the field $\C$ of complex numbers.

\subsection{Generalized pairs}

A \textit{generalized pair} $(X,B+M)$ consists of
\begin{enumerate}[(1)]
  \item a normal projective variety $X$,
  \item an effective $\R$-divisor $B$ on $X$, and
  \item a projective morphism $f\colon Y\rightarrow X$ from a normal variety $Y$ and a nef $\R$-Cartier divisor $M_Y$ on $Y$ such that $M\coleq f_{\ast}M_Y$ and $K_X+B+M$ is $\R$-Cartier.
\end{enumerate}

Note that if $M_Y=M=0$, then this notion coincides with the usual pair and we drop the word ``generalized'' or the prefix ``g''.

\begin{definition}
  Let $(X,B+M)$ be a generalized pair with data $f\colon Y\rightarrow X$ and $E$ a prime divisor over $X$. Replacing $Y$ by a higher birational model if necessary, we may assume that $E$ is a divisor on $Y$. If we write $K_Y+B_Y+M_Y=f^{\ast}(K_X+B+M)$ for some $\R$-divisor $B_Y$ on $Y$, then the \textit{generalized log discrepancy} $a(E;X,B+M)$ of $E$ with respect to $(X,B+M)$ is defined as
  \begin{align*}
    a(E;X,B+M)\coleq 1-\mult_E B_Y.
  \end{align*}
We say that a generalized pair $(X,B+M)$ is a \textit{gklt} (resp. \textit{glc}) pair if $a(E;X,B+M)>0$ (resp. $\geq 0$) holds for any prime divisor $E$ over $X$. For a prime divisor $E$ over $X$, the \textit{center} $C_X(E)$ is the image of the prime divisor $E$ on $X$. The \textit{generalized non-klt locus} $\gnklt(X,B+M)$ of a generalized pair $(X,B+M)$ is defined as
  \begin{align*}
    \gnklt(X,B+M)\coleq \bigcup_E C_X(E),
  \end{align*}
where the union is taken over all prime divisors $E$ over $X$ such that $a(E;X,B+M)\leq 0$. For a glc pair $(X,B+M)$, we call such $C_X(E)$ the \textit{glc center} of $(X,B+M)$.
\end{definition}

\subsection{Minimal model program for generalized pairs}

Let $X$ be a normal projective variety and $D$ a pseudoeffective $\R$-Cartier divisor on $X$. A birational contraction $\varphi\colon X\dashrightarrow Y$ is called \textit{$D$-nonpositive} if $\varphi_*D$ is an $\R$-Cartier divisor on a normal projective variety $Y$, and there is a common resolution $(p,q)\colon Z\rightarrow X\times Y$ such that
	\begin{align*}
		p^\ast D=q^\ast \varphi_\ast D+E,
	\end{align*}
where $E$ is an effective $q$-exceptional divisor. The map $\varphi$ is called \textit{$D$-negative} if additionally, $\Supp(E)$ contains all the strict transforms of the $\varphi$-exceptional divisors.

A typical $D$-negative contraction is the map given by the \textit{$D$-minimal model program ($D$-MMP for short)}.

\begin{definition}
  Let $X$ be a normal projective variety and $D$ a pseudoeffective $\R$-Cartier divisor on $X$. A birational map $\varphi\colon X\dashrightarrow Y$ is called a $D$-MMP or $D$-minimal model if $\varphi$ is a $D$-negative contraction and $\varphi_{\ast}D$ is nef on $Y$.
\end{definition}

Let us recall the notion of $D$-minimal model program. Let $X$ be a $\Q$-factorial normal projective variety. Let $\varphi\colon X\rightarrow Y$ be a birational contraction morphism with $\rho(X/Y)=1$. If the exceptional locus $\Exc(\varphi)$ is a prime divisor, then the map $\varphi$ is called a \textit{divisorial contraction}. If $\codim\Exc(\varphi)\geq 2$, then $\varphi$ is called a \textit{small contraction}.

Let $\overline{\mathrm{NE}}(X)$ be the Mori cone of $X$ and $R$ a $D$-negative extremal ray of $\overline{\mathrm{NE}}(X)$. We say that a contraction morphism $\varphi\colon X\rightarrow Y$ is a \textit{$D$-negative extremal contraction} if $\rho(X/Y)=1$ and $\varphi(C)=\mathrm{pt}$ for all curves $C$ such that $R=\R_{\geq0}[C]$. Suppose that $\varphi\colon X\rightarrow Y$ is a small $D$-negative extremal contraction. Then $\varphi$ is called a \textit{$D$-flipping contraction}. Moreover, we can define a \textit{$D$-flip} $\varphi^{+}\colon X^{+}\rightarrow Y$ of a flipping contraction $\varphi\colon X\rightarrow Y$. For the definition of $D$-flip, we refer to \cite[Definition 2.3]{Bir}.

The $D$-MMP is a sequence of birational $D$-negative contractions;
\begin{align*}
  X\coleq X_1\overset{f_1}{\dashrightarrow} X_2\overset{f_2}{\dashrightarrow} X_3\overset{f_3}{\dashrightarrow} \cdots
\end{align*}
where each $f_{i}\colon X_{i}\dashrightarrow X_{i+1}$ is either a divisorial contraction or a flip. The goal of the $D$-MMP is to find either a $D$-minimal model or a $D$-Mori fibre space, i.e., $f\colon X\rightarrow Y$ is a contraction morphism with $\rho(X/Y)=1$, $\dim Y<\dim X$ and $-D$ is $f$-ample.

We say \textit{MMP for $(X,B+M)$} or \textit{MMP on $(K_X+B+M)$} interchangeably to mean \textit{$(K_X+B+M)$-MMP}. Since the MMP for glc generalized pairs is established in \cite{CHLX}, we can run the MMP on any $\Q$-factorial glc pair $(X,B+M)$.

\begin{theorem}\cite[Theorem 2.2.3]{CHLX}\label{thm:CHLX}
  Let $(X,B+M)$ be a $\Q$-factorial glc pair. Then there exists a sequence of $(K_X+B+M)$-negative contractions $f\colon X\dashrightarrow Y$ which consists of divisorial contractions and flipping contractions. Moreover, this sequence terminates with a Mori fibre space or a minimal model, or the sequence is eventually an infinite sequence of flips.
\end{theorem}

\subsection{Asymptotic valuation and Zariski decomposition}
Let $Y$ be a smooth projective variety and $D$ a big $\R$-Cartier divisor on $Y$. For a prime divisor $E$ on $Y$, we define the \textit{asymptotic valuation} $\sigma_E (D)$ of $D$ along $E$ as
\begin{align*}
  \sigma_E(D)\coleq \inf\{\mult_E D'\mid D\sim_{\R}D'\geq 0\}.
\end{align*}
The asymptotic valuation extends to pseudoeffective divisors. If $D$ is only pseudoeffective, then we define $\sigma_E(D)\coleq \lim\limits_{\varepsilon\rightarrow 0} \sigma_E(D+\varepsilon A)$ for an ample divisor $A$ on $X$. It is well known that this definition is independent of the choice of an ample divisor $A$, and there are only finitely many prime divisors $E$ on $Y$ such that $\sigma_E(D)>0$. See \cite{Nak} for details.

Let $X$ be a normal projective variety and $D$ a pseudoeffective $\R$-Cartier divisor on $X$. For a prime divisor $E$ over $X$, the asymptotic valuation $\sigma_E(D)$ is defined as $\sigma_E(D)\coleq \sigma_E(f^{\ast}D)$ for a resolution $f\colon Y\rightarrow X$. By \cite[Theorem III.5.16]{Nak}, $\sigma_E(D)$ is independent on the choice of a resolution $f$.

The \textit{negative part} $\np(D)$ of $D$ is defined as
\begin{align*}
  \np(D)\coleq \sum_E \sigma_E(D)E,
\end{align*}
and the \textit{positive part} $\pp(D)$ of $D$ is defined as $\pp(D)\coleq D-\np(D)$. We call $D=\pp(D)+\np(D)$ the \textit{divisorial Zariski decomposition} of $D$. We call it \textit{Zariski decomposition} if $\pp(D)$ is nef. Moreover, if there exists a projective birational morphism $f\colon Y\rightarrow X$ such that $\pp(f^{\ast}D)$ is nef, then we say that $D$ admits a \textit{birational Zariski decomposition}. If $\pp(f^{\ast}D)$ is $\mathrm{NQC}$, i.e., a linear combination of $\Q$-Cartier nef divisors with nonnegative coefficients, then we say that $D$ admits a birational Zariski decomposition \textit{with $\mathrm{NQC}$ positive part}.

\subsection{Asymptotic base loci}

Let $X$ be a normal projective variety and $D$ an $\R$-Cartier divisor on $X$. The \textit{stable base locus} $\B(D)$ of $D$ is defined as
\begin{align*}
  \B(D)\coleq \bigcap\{\Supp(D')\mid D\sim_{\R} D'\geq 0\}.
\end{align*}
Variations of the stable base locus were defined and studied in \cite{ELMNP-base loci}. The \textit{augmented base locus} $\bp(D)$ of $D$ is defined as
\begin{align*}
  \bp(D)\coleq \bigcap_{A} \B(D-A),
\end{align*}
where the intersection is taken over all ample $\R$-divisors $A$ on $X$. The \textit{restricted base locus} (or \textit{diminished base locus}) $\bm(D)$ of $D$ is defined as
\begin{align*}
  \bm(D)\coleq \bigcup_{A} \B(D+A),
\end{align*}
where the union is taken over all ample $\R$-divisors $A$ on $X$.

The \textit{non-nef locus} $\nnef(D)$ of $D$ is defined as
\begin{align*}
  \nnef(D)\coleq \bigcup_E C_X(E),
\end{align*}
where the union is taken over all prime divisors $E$ over $X$ such that $\sigma_E(D)>0$. In general, we have the following inclusions
\begin{align*}
  \nnef(D)\subseteq \bm(D)\subseteq \B(D)\subseteq \bp(D).
\end{align*}
Also, if $D=P+N$ is the divisorial Zariski decomposition, then $\Supp(N)\subseteq \nnef(D)$. Unlike $\B(D)$ and $\bp(D)$, the restricted base locus $\bm(D)$ is not Zariski closed in general \cite{Les}.

\subsection{Potential triples}

A \textit{potential triple} $(X,\Delta,D)$ consists of a pair $(X,\Delta)$ and a pseudoeffective $\R$-Cartier divisor $D$ on $X$.

\begin{definition}
  Let $(X,\Delta,D)$ be a potential triple and $E$ a prime divisor over $X$. We define the \textit{potential log discrepancy} $a(E;X,\Delta,D)$ of $(X,\Delta,D)$ with respect to $E$ as
  \begin{align*}
    a(E;X,\Delta,D)\coleq a(E;X,\Delta)-\sigma_E (D)
  \end{align*}
  where $a(E;X,\Delta)$ is the log discrepancy of $(X,\Delta)$ with respect to $E$. A potential triple $(X,\Delta,D)$ is called a \textit{pklt} (resp. \textit{plc}) triple if $\inf\limits_E a(E;X,\Delta,D)>0$ (resp. $\geq 0$), where the infimum is taken over all prime divisors $E$ over $X$. When $D=-(K_X+\Delta)$ and it is pseudoeffective, we drop ``$D$" from $(X,\Delta,D)$ and call $(X,\Delta)$ a \textit{potential pair}, which was first introduced and studied in \cite{CP}.
\end{definition}

As we can see in the following theorem, the potential discrepancy is nondecreasing along $(K_X+\Delta+D)$-negative contractions.

\begin{lemma}[\protect{cf. \cite[Theorem 2.6]{CJK}}]\label{lem:pot discrep}
Let $(X,\Delta,D)$ be a potential triple and $R$ a $(K_{X}+\Delta+D)$-negative extremal ray of $\overline{\NE}(X)$. Assume that we have either a divisorial contraction or a flip $\varphi\colon X\dashrightarrow X'$ associated to the ray $R$. Let $\Delta'\coleq\varphi_*\Delta$ and $D'\coleq\varphi_*D$. Then we have
$$a(E;X,\Delta,D)\leq a(E;X',\Delta',D') $$
for any prime divisor $E$ over $X$.
\end{lemma}

\begin{proof}
Let $p\colon W\to X$ and $q\colon W\to X'$ be a common resolution. Then we can write
$$ p^{\ast}(K_{X}+\Delta+D)=q^{\ast}(K_{X'}+\Delta'+D')+\Phi,$$
where $\Phi$ is an effective $q$-exceptional divisor. First, assume that $D\cdot R>0$. By the negativity lemma, we can find an effective $q$-exceptional divisor $\Phi_{1}$ such that $p^{\ast}D+\Phi_{1}=q^{\ast}D'$. Thus, for a prime divisor $E$ on $W$, we have
$$ \sigma_{E}(D')\le \sigma_{E}(D)+\mult_{E}\Phi_{1}. $$
Let $\Phi_{2}=p^{\ast}(K_{X}+\Delta)-q^{\ast}(K_{X'}+\Delta')$. Then we have
$$ \Phi_{2}=\sum_{E}(a(E;X',\Delta')-a(E;X,\Delta))E, $$
where $E$ runs over all prime divisors on $W$. By construction, $\Phi=\Phi_{2}-\Phi_{1}$ and we obtain that
\begin{align*}
a(E;X',\Delta',D')-a(E;X,\Delta,D)&=a(E;X',\Delta')-a(E;X,\Delta)+\sigma_{E}(D)-\sigma_{E}(D')\\
&\ge \mult_{E}\Phi_{2}-\mult_{E}\Phi_{1}\\
&=\mult_{E}\Phi\ge 0.
\end{align*}

Now, assume that $D\cdot R\le 0$. Again, by the negativity lemma, we can find an effective $q$-exceptional divisor $\Phi_{1}$ such that $p^{\ast}D=q^{\ast}D'+\Phi_{1}$. In this case, by \cite[Lemma III.5.14]{Nak}, we have
$$ \sigma_{E}(D)=\sigma_{E}(D')+\mult_{E}\Phi_{1} $$
for any prime divisor $E$ on $W$. Let $\Phi_{2}=p^{\ast}(K_{X}+\Delta)-q^{\ast}(K_{X'}+\Delta')$. Then $\Phi=\Phi_{1}+\Phi_{2}$ and we obtain that
\begin{align*}
a(E;X',\Delta',D')-a(E;X,\Delta,D)&=a(E;X',\Delta')-a(E;X,\Delta)+\sigma_{E}(D)-\sigma_{E}(D')\\
&=\mult_{E}\Phi_{2}+\mult_{E}\Phi_{1}\\
&=\mult_{E}\Phi\ge 0.\qedhere
\end{align*}
\end{proof}

In particular, if $(X,\Delta,D)$ is a plc triple and $f\colon (X,\Delta,D)\dashrightarrow (Y,\Delta_Y,D_Y)$ is a $(K_X+\Delta+D)$-negative contraction, then the resulting triple $(Y,\Delta_Y,D_Y)$ is also a plc triple.

Note that we require $K_X+B+M$ to be only $\R$-Cartier for a generalized pair $(X,B+M)$. 
On the other hand, in a potential triple $(X,\Delta,D)$, $K_X+\Delta$ and $D$ are both $\R$-Cartier divisors. Even if the underlying variety $X$ is $\Q$-factorial, the notions of pklt and plc are more general than those of gklt and glc.

\begin{proposition} [\protect{\cite[Remark 3.1]{CJK}}]\label{prop:g implies p}
  Let $(X,\Delta+D)$ be a generalized pair and we consider $(X,\Delta,D)$ as a potential triple. If $(X,\Delta+D)$ is a $\Q$-factorial gklt (resp. glc) pair, then $(X,\Delta,D)$ is a pklt (resp. plc) triple.
\end{proposition}

As mentioned in \cite{CJK}, the collection of $\Q$-factorial potential triples is strictly larger than that of generalized pairs. The main reason is that, for a potential triple $(X,\Delta,D)$, $D$ does not necessarily admit a birational Zariski decomposition. Below are examples of divisors which do not admit any (birational) Zariski decomposition.

\begin{example}\hfill\label{Example}
  \begin{enumerate}[(1)]
    \item Let $X$ be a smooth projective variety and $D$ a pseudoeffective divisor on $X$. For a prime divisor $E$ on $X$, we have $\sigma_E(D)>0$ if and only if $E\subset \bm(D)$. Hence, if $\codim\B(D)\geq 2$ (i.e., $D$ is a movable divisor), then $N_\sigma(D)=0$. If additionally $D$ is not nef, then $D$ does not have the Zariski decomposition.
    \item Let $X$ be a normal projective variety and $D$ a pseudoeffective $\R$-divisor on $X$. It is well known that if there exists a resolution $f\colon Y\rightarrow X$ for which $D$ admits a birational Zariski decomposition $f^{\ast}D=P+N$, then $\bm(D)$ is Zariski closed. In \cite{Les}, Lesieutre constructed an effective divisor $D$ on a three-dimensional variety $X$ such that $\bm(D)$ is not Zariski closed. Hence, $D$ does not admit a birational Zariski decomposition.
    \item The bigness of a divisor $D$ or Zariski closedness of $\bm(D)$ do not guarantee the existence of a birational Zariski decomposition. In \cite[Chapter IV, \textsection 2]{Nak}, Nakayama constructed a big divisor $D$ on a four-dimensional variety $X$ such that $D$ does not admit a birational Zariski decomposition. In this example, however, $\bm(D)$ is Zariski closed.
  \end{enumerate}
  If $(X,\Delta,D)$ is a potential triple with $D$ as in (2) or (3), then one cannot associate a generalized pair structure.
\end{example}

\begin{definition}
  Let $(X,\Delta,D)$ be a potential triple. The \textit{potentially non-klt locus} $\pnklt(X,\Delta,D)$ of $(X,\Delta,D)$ is defined as
  \begin{align*}
    \pnklt(X,\Delta,D)\coleq \bigcup_E C_X(E),
  \end{align*}
  where the union is taken over all prime divisor $E$ over $X$ such that $a(E;X,\Delta,D)\leq 0$. For a plc triple $(X,\Delta,D)$, such an image $C_X(E)$ is called the \textit{plc center} of $(X,\Delta,D)$.
\end{definition}

Note that we have the following inclusions
\begin{align*}
  \nklt(X,\Delta)\subseteq \pnklt(X,\Delta,D)\subseteq \nklt(X,\Delta)\cup \nnef(D).
\end{align*}
By \cite[Theorem 1.4]{CJK}, if $(X,\Delta,D)$ is a potential triple such that $D$ is a big $\Q$-divisor, then $\pnklt(X,\Delta,D)$ is Zariski closed. However, in general, unlike the generalized non-klt locus, it is unclear whether $\pnklt(X,\Delta,D)$ is Zariski closed or not.

\subsection{MMP for potential triples}\label{subsec: pklt,glc}
We run the MMP on potential triples as follows. Let $(X,\Delta,D)$ be a $\Q$-factorial plc triple. If $K_X+\Delta+D$ is nef, then $(X,\Delta,D)$ is a minimal model and there is nothing further to do. If $K_X+\Delta+D$ is not nef, then using Theorem \ref{thm:main1} we construct a sequence of divisorial contractions and flips that are $(K_X+\Delta+D)$-negative. By Lemma \ref{lem:pot discrep}, the triples in the intermediate steps are all plc triples. In this manner, after finite steps we expect to obtain either a $(K_X+\Delta+D)$-minimal model or a $(K_X+\Delta+D)$-Mori fibre space. As we have emphasized in Example \ref{Example}, the category of potential triples is strictly larger than that of generalized pairs. Hence, we expect that we would obtain more delicate classification of varieties once the MMP on potential triples is established.

In some special cases as in Theorem \ref{thm:main2} and Theorem \ref{thm:run mmp}, we can run the $(K_X+\Delta+D)$-MMP by reducing to the MMP on generalized pairs. However, in order to obtain such MMP on plc triples in general, we will need to require the Cone theorem and Contraction theorem for potential triples. These will appear in our future works.

\section{Main results and Proofs}\label{sect:main}

In this section, we prove various properties of potential triples. It is well known that the non-klt locus $\nklt(X,\Delta)$ of a pair $(X,\Delta)$ is closed since it can be computed by a fixed log resolution of $(X,\Delta)$. In \cite[Theorem 2.6]{CJ}, when $D=-(K_X+\Delta)$ is pseudoeffective in a potential triple and $-(K_X+\Delta)$ admits a birational Zariski decomposition, then $\pnklt(X,\Delta,D)$ is also closed by the same reason that it can also be computed by some fixed log resolution of $(X,\Delta)$.

For a plc triple $(X,\Delta,D)$ with a pseudoeffective divisor $D$ which admits a birational Zariski decomposition, one can associate a glc pair structure to $(X,\Delta,D)$ by the following theorem. As a byproduct, we obtain that the potentially non-klt locus $\pnklt(X,\Delta,D)$ is Zariski closed.

\begin{theorem}\label{thm:generalized}
Let $(X,\Delta,D)$ be a potential triple with a pseudoeffective divisor $D$ which admits a birational Zariski decomposition $f^{\ast}D=P+N$ for some projective birational morphism $f\colon Y\to X$. Then $(X,\Delta,D)$ is a pklt (resp. plc) triple if and only if $(X,(\Delta+f_{\ast}N)+f_{\ast}P)$, considered as a generalized pair, is a gklt (resp. glc) pair. Furthermore, we have $\pnklt(X,\Delta,D)=\gnklt(X,(\Delta+f_{\ast}N)+f_{\ast}P)$, which implies that $\pnklt(X,\Delta,D)$ is Zariski closed.
\end{theorem}

\begin{proof}
If we write $K_{Y}+\Delta_{Y}=f^{\ast}(K_{X}+\Delta)$ for some $\R$-divisor $\Delta_Y$ on $Y$, then we have $K_{Y}+\Delta_{Y}+N+P=f^{\ast}(K_{X}+\Delta+D)$. Let $E$ be a prime divisor over $X$. By taking $Y$ higher if necessary, we can assume that $E$ is a divisor on $Y$. By definition, we have the following equalities 
$$a(E;X,\Delta,D)=a(E;X,\Delta)-\mult_{E}N=1-\mult_{E}(\Delta_Y+N).$$ 
Hence, $(X,(\Delta+f_{\ast}N)+f_{\ast}P)$ is a gklt (resp. glc) pair if and only if $(X,\Delta,D)$ is a pklt (resp. plc) triple. Moreover, any plc center of $(X,\Delta,D)$ is a glc center of $(X,(\Delta+f_{\ast}N)+f_{\ast}P)$. Therefore, we can conclude that $\pnklt(X,\Delta,D)=\gnklt(X,(\Delta+f_{\ast}N)+f_{\ast}P)$. In particular, since the generalized non-klt locus is Zariski closed, so is $\pnklt(X,\Delta,D)$.
\end{proof}

By letting $D=-(K_X+\Delta)$ in Theorem \ref{thm:generalized}, the following corollary immediately follows.

\begin{corollary}[cf. \protect{\cite[Theorem 2.6]{CJ}}]
  Let $(X,\Delta)$ be a plc pair with pseudoeffective $-(K_X+\Delta)$. If $-(K_X+\Delta)$ admits a birational Zariski decomposition, then $\pnklt(X,\Delta)$ is Zariski closed.
\end{corollary}

\begin{theorem}\label{thm:run mmp}
  Let $(X,\Delta,D)$ be a $\Q$-factorial plc triple such that $D$ admits a birational Zariski decomposition. Then we can run a $(K_X+\Delta+D)$-MMP.
\end{theorem}
\begin{proof}
  Let $f^\ast D=P+N$ be the birational Zariski decomposition for some projective birational morphism $f\colon Y\rightarrow X$. By Theorem \ref{thm:generalized}, $(X,(\Delta+f_{\ast}N)+f_{\ast}P)$ is a glc pair. Hence, by Theorem \ref{thm:CHLX}, we can run a generalized MMP on $K_X+(\Delta+f_{\ast}N)+f_{\ast}P\equiv K_X+\Delta+D$.
\end{proof}

Now we have the proof of our first main theorem of the paper.

\begin{proof}[Proof of Theorem \ref{thm:main1}]
Theorems \ref{thm:generalized} and \ref{thm:run mmp} imply the statements (1) and (2) in Theorem \ref{thm:main1}.
\end{proof}

By Theorem \ref{lem:pot discrep}, the plc condition is preserved along this $(K_X+\Delta+D)$-MMP. We note that the termination of  $(K_X+\Delta+D)$-MMP is not guaranteed.

\begin{corollary}\label{cor:-K lc mmp}
  Let $(X,\Delta,D)$ be a $\Q$-factorial  plc triple with $D\coleq -(1+\varepsilon)(K_X+\Delta)$ pseudoeffective for some small positive $\varepsilon$. If $D$ admits a birational Zariski decomposition, then we can run the $-(K_X+\Delta)$-MMP.
\end{corollary}

By detouring through generalized pairs, we give another proof of the main theorem of \cite{Jan}. As noted in \cite[Remark 5.19]{TX}, the $\mathrm{NQC}$ condition is necessary in \cite[Theorem 5.18]{TX}.

\begin{proof}[Proof of Corollary \ref{cor:main1}]
  Let $-f^{\ast}(K_X+\Delta)=P+N$ be the birational Zariski decomposition for some projective birational morphism $f\colon Y\rightarrow X$. By Theorem \ref{thm:generalized}, $(X,(\Delta+f_{\ast}N)+f_{\ast}P)$ is a gklt pair. For a sufficiently small rational number $\varepsilon>0$, $(X,(\Delta+(1+\varepsilon)f_{\ast}N)+(1+\varepsilon)f_{\ast}P)$ is also a gklt pair. By Theorem \ref{thm:generalized}, $(X,\Delta,D\coleq -(1+\varepsilon)(K_X+\Delta))$ is a pklt triple. Moreover, $K_X+(\Delta+(1+\varepsilon)f_{\ast}N)+(1+\varepsilon)f_{\ast}P\equiv K_X+\Delta+D\equiv -\varepsilon(K_X+\Delta)$ admits a birational Zariski decomposition with $\mathrm{NQC}$ positive part. Hence, by \cite[Theorem 5.18]{TX}, there exists a $(K_X+\Delta+D)$-minimal model which is also a $-(K_X+\Delta)$-minimal model.
\end{proof}

In \cite{CJL}, we showed that for a plc pair $(X,\Delta)$ with $-(K_X+\Delta)$ a big $\Q$-Cartier divisor, if no plc centers of $(X,\Delta)$ are contained in $\bp(-(K_X+\Delta))$, then there exists a good $-(K_X+\Delta)$-minimal model. It is natural to ask what would happen if we generalize this to potential triple setting.

\begin{proof}[Proof of Theorem \ref{thm:main2}]
  We note that $\nnef(D)\subseteq \bp(D)$ and the following inclusion
  $$\nklt(X,\Delta)\subseteq \pnklt(X,\Delta,D)\subseteq \nklt(X,\Delta)\cup \nnef(D).$$ Since no plc centers of $(X,\Delta,D)$ are contained in $\bp(D)$, we obtain $\pnklt(X,\Delta,D)=\nklt(X,\Delta)$. Hence, for any prime divisor $E$ over $X$, we have $a(E;X,\Delta,D)=0$ if and only if $a(E;X,\Delta)=\sigma_E(D)=0$. Since $D$ is a big $\Q$-divisor, as we can see in \cite[Proof of Proposition 4.9]{CJK}, there exists an effective divisor $D'\sim_{\Q} D$ such that $\pnklt(X,\Delta,D)=\nklt(X,\Delta+D')$. Thus, $(X,\Delta+D')$ is an lc pair and by Theorem \ref{thm:CHLX}, we can run the MMP on $K_X+\Delta+D'\equiv K_X+\Delta+D$.
\end{proof}

Even under the conditions as in Theorem \ref{thm:main2}, we do not know the existence of $(K_X+\Delta+D)$-minimal models unless $D=-(K_X+\Delta)$. When $D=-(K_X+\Delta)$, there exists a good $-(K_X+\Delta)$-minimal model by \cite[Theorem 1.2]{CJL}. By Theorem \ref{lem:pot discrep}, the plc condition is also preserved along this MMP.

Let $X_{\nklt}\coleq \bigcap \nklt(X,\Delta)$ where the intersection is taken over all pairs $(X,\Delta)$ with effective $\Q$-divisors $\Delta$. By \cite[Corollary 4.7]{CD}, $\nnef(D)$ and $\bm(D)$ coincide outside $X_{\nklt}$. However, if $D$ has good positivity, then the loci $\nnef(D), \bm(D)$ and $\B(D)$ all coincide without assuming any singularity conditions on $X$. The following proposition was first proved in \cite[Lemma 3.1]{TX2} under slightly different conditions. However, the result and the proof is almost verbatim.

\begin{proposition}[cf. \protect{\cite[Lemma 3.1]{TX2}}]\label{prop:loci coincide}
  Let $X$ be a normal projective variety and $D$ an $\R$-Cartier divisor on $X$. Suppose that $f\colon Y\rightarrow X$ is a projective birational morphism from a normal projective variety $Y$. If $D$ admits a birational Zariski decomposition $f^\ast D=P+N$ such that $P$ is semiample, then $\nnef(D)=\bm(D)=\B(D)$.
\end{proposition}
\begin{proof}
  We first note that the following inclusions always hold
  \begin{align*}
    \nnef(D)\subseteq \bm(D)\subseteq \B(D).
  \end{align*}
  Let us first show that $\nnef(D)=f(\Supp(N))$. If $N=\sum a_i N_i$ for some positive numbers $a_i$, then $\sigma_{N_i}(D)= a_i>0$. Therefore, we have the inclusion
  $$\Supp(N)\subseteq \nnef(f^\ast D).$$
  On the other hand, for any ample divisor $A$ on $Y$, we have $\B(f^\ast D+A)\subseteq \Supp(N)$ since $\B(P+A)=\emptyset$. Moreover, by the definition of restricted base locus, we have $\bm(f^\ast D)=\bigcup\limits_A \B(f^\ast D+A)$, where the union is taken over all ample divisors $A$. Thus, we have that
  $$\bm(f^\ast D)\subseteq \Supp(N).$$
  Therefore, we obtain the equality $\nnef(f^\ast D)=\Supp(N)$, and we deduce that $\nnef(D)=f(\Supp(N))$ since $\nnef(D)=f(\nnef(f^\ast D))$.

  Now if we assume that $P$ is semiample, then we have $\B(f^\ast D)\subseteq \Supp(N)$ and $\B(D)=\B(f^\ast D)$. Thus, we obtain the inclusion $\B(D)\subseteq f(\Supp(N))$. Consequently, we obtain the required result.
\end{proof}

In particular, under the semiampleness condition as in Proposition \ref{prop:loci coincide}, the restricted base locus is Zariski closed.

\begin{corollary}\label{cor:bm closed}
  Let $(X,\Delta)$ be a plc pair. If $(X,\Delta)$ admits a good $-(K_X+\Delta)$-minimal model, then we have $\nnef(-(K_X+\Delta))=\bm(-(K_X+\Delta))=\B(-(K_X+\Delta))$. In particular, $\bm(-(K_X+\Delta))$ is Zariski closed.
\end{corollary}
\begin{proof}
  By assumption, there exists a birational contraction $\varphi\colon (X,\Delta)\dashrightarrow (Y,\Delta_Y)$ such that $-(K_Y+\Delta_Y)$ is semiample. Hence, one can see that $-(K_X+\Delta)$ admits a birational Zariski decomposition with semiample positive part. Thus, by Theorem \ref{prop:loci coincide}, we complete the proof.
\end{proof}


\end{document}